\documentclass[12pt]{article}
\usepackage{amsmath,amssymb,amsbsy,amsfonts,amsthm,latexsym,
                                amsopn,amstext,amsxtra,euscript,amscd}

\newtheorem{prop}{Proposition}

\newtheorem{thm}{Theorem}

\DeclareMathOperator{\Card}{Card}

\def\({\left(}
\def\){\right)}
\def\[{\left[}
\def\]{\right]}
\def\<{\langle}
\def\>{\rangle}

\def\la{\lambda}

\begin{document}

\title{Sums of the digits in bases $2$ and $3$}

\author{Jean-Marc Deshouillers, Laurent Habsieger, \\Shanta Laishram, Bernard Landreau}

\pagenumbering{arabic}

\date{\today}

\maketitle

\begin{flushright}
\emph{To Robert Tichy, for his $60$th birthday}
\end{flushright}

\begin{abstract}
Let $b \ge 2$ be an integer and let $s_b(n)$ denote the sum of the digits of the representation of an integer $n$ in base $b$.  For sufficiently large $N$, one has
\begin{equation}\notag
\Card \{n \le N : \left|s_3(n) - s_2(n)\right| \le 0.1457205 \log n \} \, > \, N^{0.970359}.
\end{equation}
The proof only uses the separate (or marginal) distributions of the values of $s_2(n)$ and $s_3(n)$.
\end{abstract}

\section{Introduction} 
For integers $b \ge 2$ and $n\ge 0$, we denote by ``the sum of the digits of $n$ in base $b$" the quantity
\begin{equation}\notag
s_b(n) = \sum_{j \ge 0} \varepsilon_j, \text{ where } n = \sum_{j \ge 0} \varepsilon_j b^j \text{ with } \forall j : \varepsilon_j \in \{0, 1, \ldots, b-1\}.
\end{equation}

Our attention on the question of the proximity of $s_2(n)$ and $s_3(n)$ comes from the apparently non related question of the distribution of the last non zero digit of $n!$ in base $12$ (cf. \cite{DR} and \cite{De}).\footnote{Indeed, if the last non zero digit of $n!$ in base $12$ belongs to $\{1, 2, 5, 7, 10, 11\}$ then $|s_3(n) - s_2(n)| \le 1$; this seems to occur infinitely many times.}\\

Computation shows that there are $48~266~671~607$ positive integers up to $10^{12}$ for which $s_2(n) = s_3(n)$, but it seems to be unknown whether there are infinitely many integers $n$ for which $s_2(n) = s_3(n)$ or even for which $|s_2(n) - s_3(n)|$ is significantly small.\\

We do not know the first appearance of the result we quote as Theorem \ref{almostall}; in any case, it is a straightforward application of the fairly general main result of N. L. Bassily and  I. K\'{a}tai \cite{BK}. We recall that a sequence $\mathcal{A} \subset \mathbb{N}$ of integers is said to have asymptotic natural density $1$ if
$$
\Card \{n \le N : n \in \mathcal{A} \} = N + o(N).
$$

\begin{thm}\label{almostall}
Let $\psi$ be a function tending to infinity with its argument. The sequence of natural numbers $n$ for which
\begin{eqnarray}\notag
\left(\frac{1}{\log 3} - \frac{1}{\log 4} \right) \log n - \psi(n) \sqrt{\log n} &\le& s_3(n) - s_2(n)\\
\notag
&\le& \left(\frac{1}{\log 3} - \frac{1}{\log 4} \right) \log n + \psi(n) \sqrt{\log n}
\end{eqnarray}
has asymptotic natural density $1$.
\end{thm}

Our main result is that there exist infinitely many $n$ for which $\left|s_3(n) - s_2(n) \right|$ is significantly smaller than $\left(\frac{1}{\log 3} - \frac{1}{\log 4} \right) \log n = 0.18889... \log n$. More precisely we have the following:

\begin{thm}\label{main}
For sufficiently large $N$, one has
\begin{equation}\label{eqmain}
\Card \{n \le N : \left|s_3(n) - s_2(n)\right| \le 0.1457205 \log n \} \, > \, N^{0.970359}. 
\end{equation}
\end{thm}

The mere information we use in proving Theorem \ref{main} is the knowledge of the separate (or marginal) distributions of $\left(s_2(n)\right)_n$ and $\left(s_3(n)\right)_n$, without using any further information concerning their joint distribution.\\

In Section 2, we provide a heuristic approach to Theorems \ref{almostall} and \ref{main}; the actual distribution of $\left(s_2(n)\right)_n$and $\left(s_3(n)\right)_n$ is studied in Section 3. The proof of Theorem \ref{main} is given in Sections 4.\\

Let us formulate three remarks as a conclusion to this introductory section.\\
\indent It seems that our present knowledge of the joint distribution of $s_2$ and $s_3$ (cf. for exemple C. Stewart \cite{St} for a Diophantine approach or M. Drmota \cite{Dr} for a probabilistic one) does not permit us to improve on Theorem \ref{main}.\\
\indent Theorem \ref{main} can be extended to any pair of distinct bases, say $q_1$ and $q_2$: more than computation, the Authors have deliberately chosen to present an idea to the Dedicatee.\\
\indent Although we could not prove it, we believe that Theorem \ref{main} represents the limit of our method.\\

\noindent \textbf{Acknowledgements}  The authors are indebted to Bernard Bercu for several discussions on the notion of ``spacing" between two random variables, a notion to be developped later. They also thank the Referees for their constructive comments. The first, third and fourth authors wish to thank the Indo-French centre CEFIPRA for the support permitting them to collaborate on this project (ref. 5401-A). The first named author acknowledges with thank the support of the French-Austrian project MuDeRa (ANR and FWF).\\

\section{A heuristic approach}
As a warm-up for the actual proofs, we sketch a heuristic approach. A positive integer $n$ may be expressed as
$$
n= \sum_{j=0}^{J(n)} \varepsilon_j(n) b^j, \text{ with } J(n)= \left\lfloor\frac{\log n}{\log b}\right\rfloor.
$$
If we consider an interval of integers around $N$, the smaller is $j$ the more equidistributed are the $\varepsilon_j(n)$'s, and the smaller are the elements of a family $\mathcal{J}=\{j_1 < j_2 < \cdots < j_s\}$ the more independent are the $\varepsilon_j(n)$'s for $j \in \mathcal{J}$. Thus a first model for $s_b(n)$ for $n$ around $N$ is to consider a sum of $\left\lfloor\frac{\log N}{\log b}\right\rfloor$ independent random variables uniformly distributed in $\{0, 1, \ldots, b-1\}$. Thinking of the central limit theorem, we even consider a continuous model, representing $s_b(n)$, for $n$ around $N$ by a Gaussian random variable $S_{b, N}$ with expectation and variance given by
$$
\mathbb{E}\left(S_{b, N}\right) = \frac{(b-1) \log N}{2 \log b} \text{ and } \mathbb{V}\left(S_{b, N}\right) = \frac{(b^2-1) \log N}{12 \log b}.
$$
In particular
$$
\mathbb{E}\left(S_{2, N}\right) = \frac{ \log N}{\log 4} \text{ and } \mathbb{E}\left(S_{3, N}\right) = \frac{\log N}{ \log 3} ,
$$
and their standard deviations have the order of magnitude $\sqrt{\log N}$.\\

\emph{Towards Theorem \ref{almostall}. }In \cite{BK}, it is proved that a central limit theorem actually holds for $s_b$; more precisely, the following proposition is the special case of the first relation in the main Theorem of \cite{BK}, with $f(n) = s_b(n)$ and $P(X)=X$.
\begin{prop}\label{B-K}
For any positive $y$, as $x$ tend to infinity, one has
\begin{equation}\notag
\frac{1}{x}\Card\left\{n < x \, : \, \left|s_b(n) - \mathbb{E}\left(S_{b, n}\right)     \right| < y\left( \mathbb{V}\left(S_{b, n}\right) \right)^{1/2}\right\} \rightarrow \frac{1}{\sqrt{2 \pi}}\int_{-y}^y e^{-t^2/2} dt.
\end{equation}
\end{prop}
\noindent Theorem \ref{almostall} easily follows from Proposition \ref{B-K}: the set under our consideration is the intersection of 2 sets of density 1.\\

\emph{Towards Theorem \ref{main}. }If we wish to deal with a difference $|s_3(n) - s_2(n)| < u \log n$ for some $u < \left(\frac{1}{\log 3} - \frac{1}{\log 4}\right)$ we must, by what we have seen above, consider events of asymptotic probability zero, which means that a heuristic approach must be substantiated by a rigorous proof. Our key remark is that the variance of $S_{3, N}$ is larger than that of $S_{2, N}$; this implies the following: the probability that $S_{3, N}$ is at a distance $d$ from its mean is larger that the probability that $S_{2,  N}$ is at a distance $d$ from its mean. So, we have the hope to find some $u < \left(\frac{1}{\log 3} - \frac{1}{\log 4}\right)$ such that the probability that 
$\left|S_{2, N}-\mathbb{E}(S_{2, N})\right| > u \log N$ 
is smaller than the probability that $S_{3, N}$ is very close to $\mathbb{E}(S_{2, N})$. This will imply that for some $\omega$ we have $\left|S_{3,N}(\omega) - S_{2, N}(\omega)\right| \le u \log N$.\\

\section{On the distribution of the values of $s_2(n)$ and $s_3(n)$}

In order to prove Theorem \ref{main} we need\\
$\bullet$ an upper bound for the tail of the distribution of $s_2$,\\
$\bullet$ a lower bound for the tail of the distribution of $s_3$.\\

\subsection{Upper bound for the tail of the distribution of $s_2$}
\begin{prop}\label{uppertail2}
Let $\lambda \in (0, 1)$. For any 
$$
\nu > 1- \left((1-\lambda)\log(1-\lambda) + (1 + \lambda)\log(1+\lambda)\right)/\log 4   
$$
and any sufficiently large integer $H$, we have
\begin{equation}\label{equppertail2}
\Card \{ n < 2^{2H} : \left|s_2(n) - H\right| \ge \lambda H\} \le 2^{2H \nu}.
\end{equation}
\end{prop}
\begin{proof}
When $b=2$, the distribution of the values of $s_2(n)$ is simply binomial; we thus get
\begin{equation}\notag
\Card \left\{0 \le n < 2^{2H} : s_2(n) = m \right\} =  \binom{2H}{m}.
\end{equation}
Using the fact that the sequence (in $m$) $\binom{2H}{m}$ is symmetric and unimodal plus Stirling's formula, we obtain that when $m \le (1-\lambda)H$ or $m \ge (1+\lambda)H$, one has
\begin{eqnarray}\notag
\binom{2H}{m} &\le &H^{O(1)}\frac{(2H)^{2H}}{((1-\la)H)^{(1-\la)H}((1+\la)H)^{(1+\la)H}}\\
\notag
&\le& H^{O(1)}\left(\frac{2^2}{(1-\la)^{(1-\la)}(1+\la)^{(1+\la)}}\right)^H\\
\notag
&\le& H^{O(1)}\left(2^{\left(1- \left((1-\lambda)\log(1-\lambda) + (1 + \lambda)\log(1+\lambda)\right)/2\log 2    \right)} \right)^{2H}.
\end{eqnarray}
Relation (\ref{equppertail2}) comes from the above inequality and the fact that the left hand side of (\ref{equppertail2}) is the sum of at most  $2H$ such terms.
\end{proof}

\subsection{Lower bound for the tail of the distribution of $s_3$}
\begin{prop}\label{lowertail3}
Let $L$ be sufficiently large an integer. We have
\begin{equation}\label{eqlowertail3}
\Card\{n < 3^{L} : s_3(n) = \lfloor L\log3 / \log 4 \rfloor\} \ge 3^{0.970359238 L}.
\end{equation}
\end{prop}
\begin{proof}
The positive integer $L$ being given, we write any integer $n \in [0, 3^L)$ in its non necessarily proper representation, as a chain of exactly $L$ characters, $\ell_i(n) $ of them being equal to $i$, for $i \in \{0, 1, 2\}$, the sum $\ell_0 (n)+ \ell_1(n) + \ell_2(n)$ being equal to $L$, the total number of digits in this representation\footnote{For example, when $L=5$, the number "sixty" will be represented as $02020$. Happy palindromic birthday, Robert!}. One has 
\begin{equation}\label{probmodel}
\Card \left\{0 \le n < 3^L : s_3(n) = m \right\}   = \sum_{\substack{\ell_0 + \ell_1 + \ell_{2} = L \\ \ell_1 + 2 \ell_2  = m}} \frac{L!}{\ell_0 ! \ell_1 ! \ell_{2} !} .
\end{equation}
\noindent In order to get a lower bound for the left hand side of (\ref{probmodel}), it is enough to select one term in its right hand side. We choose
\begin{equation}\notag
\l_2 = \lfloor 0.235001144 L\rfloor \, ; \, \l_1 = \lfloor L \log 3 / \log 4\rfloor -2\; \l_2 \, ; \, \l_0 = L -\l_1 -\l_2.
\end{equation}
\noindent A straightforward application of Stirling's formula, similar to the one used in the previous subsection, leads to (\ref{eqlowertail3}).
\end{proof}

\section{Proof of Theorem \ref{main}}
Let $N$ be sufficiently large an integer. We let $K=\lfloor \log N / \log 3\rfloor -2$ and $H= \lfloor (K-1) \log 3 / \log 4\rfloor +2$. We notice that we have
\begin{equation}\label{HKN}
N/81 \le 3^{K-1} < 3^K < 2^{2H} \le N.
\end{equation}
We use Proposition \ref{uppertail2} with $\lambda = 0.14572049 \log 4$, which leads to
\begin{equation}\label{majo2}
\Card \{ n \le 2^{2H} : \left|s_2(n) - H\right| \ge \lambda H\} \le 2^{0.970359230\times2H} \le N^{0.970359230}.
\end{equation}
For any $n \in [2\cdot 3^{K-1}, 3^{K})$ we have $s_3(n) = 2 + s_3(n-2\cdot 3^{K-1})$ and so it follows from Proposition \ref{lowertail3} that we have
\begin{eqnarray}\notag
& & \Card\{n \in [2\cdot 3^{K-1}, 3^{K}): s_3(n) = H\} \\
\notag
& &  = \Card\{n < 3^{K-1}) : s_3(n) = H-2\}\\
 \notag
 & &= \Card\{n < 3^{K-1}) : s_3(n)= \lfloor (K-1) \log 3 / \log 4\rfloor\}\\
\notag
& &\ge 3^{0.970359238(K-1)} \ge N^{0.970359237}.
\end{eqnarray}
This implies that we have
\begin{equation}\label{mino3}
\Card \{ n \le 2^{2H} : s_3(n) = H\} \ge N^{0.970359237}.
\end{equation}
From (\ref{majo2}) and (\ref{mino3}), we deduce that for $N$ sufficiently large, we have
\begin{equation}\notag
\Card\{n \le N : | s_2(n)- s_3(n) | \le 0.1457205 \log n\} \ge N^{0.970359}.
\end{equation}
\begin{flushright}
$\Box$
\end{flushright}

\bigskip

\noindent \textbf{Addresses}\\

\noindent Jean-Marc DESHOUILLERS\\
jean-marc.deshouillers@math.u-bordeaux.fr\\
Bordeaux INP, Universit\'{e} de Bordeaux, CNRS \\
Institut Math\'{e}matique de Bordeaux, UMR 5251 \\
33405 Talence, France\\

\noindent Laurent HABSIEGER\\
habsieger@math.univ-lyon1.fr\\
Universit\'e de Lyon, CNRS UMR 5208, Universit\' e Claude Bernard Lyon 1,
Institut Camille Jordan\\ 
69622 Villeurbanne Cedex, France \\

\noindent Shanta LAISHRAM\\
shanta@isid.ac.in\\
Indian Statistical Institute\\
7 SJS Sansanwal Marg\\
110016 New Delhi, India\\

\noindent Bernard LANDREAU\\
bernard.landreau@univ-angers.fr\\
Universit\'{e} d'Angers, CNRS\\
LAREMA Laboratoire Angevin de REcherche, UMR 6093, FR 2962\\
49045 Angers, France\\

\noindent \textbf{Key words} : sums of digits, expansion in bases 2 and 3\\

\noindent \textbf{AMS 2010 Classification number}: 11K16\\

\end{document}